\documentclass[11pt]{article}
\input{amssym}
\input{epsf}
\newtheorem{prethm}{{\bf Theorem}}

\newenvironment{theorem}{\begin{prethm}{\hspace{-0.5
em}{\bf.}}}{\end{prethm}}
\newtheorem{preobs}{{\bf Observation}}

\newenvironment{obs}{\begin{preobs}{\hspace{-0.5
em}{\bf.}}}{\end{preobs}}
\newtheorem{preex}{{\bf Example}}

\newtheorem{prelem}{{\bf Theorem}}

\newenvironment{lem}{\begin{prelem}{\hspace{-0.5
em}{\bf.}}}{\end{prelem}}
\newtheorem{prepro}{{\bf Proposition}}

\newenvironment{pro}{\begin{prepro}{\hspace{-0.5
em}{\bf.}}}{\end{prepro}}
\newtheorem{preprop}{{\bf Proposition}}

\newenvironment{prop}{\begin{preprop}{\hspace{-0.5
em}{\bf.}}}{\end{preprop}}
\newtheorem{preques}{{\bf Question}}

\newenvironment{ques}{\begin{preques}{\hspace{-0.5
em}{\bf.}}}{\end{preques}}
\newtheorem{precor}{{\bf Corollary}}

\newenvironment{cor}{\begin{precor}{\hspace{-0.5
em}{\bf.}}}{\end{precor}}
\newtheorem{prelemma}{{\bf Lemma}}

\newenvironment{lemma}{\begin{prelemma}{\hspace{-0.5
em}{\bf.}}}{\end{prelemma}}
\newtheorem{presol}{{\bf Solution}}

\newtheorem{preproof}{{\bf Proof.}}

\newenvironment{proof}[1]{\begin{preproof}{\rm
               #1}\hfill{$\rule{2mm}{2mm}$}}{\end{preproof}}
\begin{document}
\title{
{\Large{\bf Characterization of  Randomly $\bf k$-Dimensional
Graphs}}}
%\\ cncm23.tex }}}
%

{\small
\author{
{\sc Mohsen Jannesari} and {\sc Behnaz Omoomi  }\\
[1mm]
{\small \it  Department of Mathematical Sciences}\\
{\small \it  Isfahan University of Technology} \\
{\small \it 84156-83111, Isfahan, Iran}}

%@@@@@@@@@@@@@@@@@@@@@@@@@@@@@@@@@@@@@@@@@@@@@@@@@@@@@@

%\subjclass[2000]{05C15.}

%\keywords{ $b$-chromatic number,  $b$-coloring, dominating coloring.\\

%\footnotetext[1]{The research  is supported by Isfahan University
%of Technology.}
 \maketitle \baselineskip15truept

\begin{abstract}
For an ordered set $W=\{w_1,w_2,\ldots,w_k\}$ of vertices and a
vertex $v$ in a connected graph~$G$, the ordered $k$-vector
$r(v|W):=(d(v,w_1),d(v,w_2),\ldots,d(v,w_k))$ is  called  the
(metric) representation of $v$ with respect to $W$, where $d(x,y)$
is the distance between the vertices $x$ and $y$. The set $W$ is
called  a resolving set for $G$ if distinct vertices of $G$ have
distinct representations with respect to $W$. A minimum resolving
set for $G$ is a basis of $G$ and its cardinality is the metric
dimension of $G$. The resolving number of a connected graph $G$ is
the minimum $k$, such that every $k$-set of vertices of $G$ is a
resolving set. A connected graph $G$ is called randomly
$k$-dimensional if each $k$-set of vertices of $G$ is a basis. In
this paper, along with some properties of randomly $k$-dimensional graphs,
 we prove that a connected graph $G$ with at least two
vertices is  randomly $k$-dimensional if and only if $G$ is
complete graph $K_{k+1}$ or an odd cycle.
\end{abstract}
{\bf Keywords:}  Resolving set; Metric dimension; Basis; Resolving
number; Basis number; Randomly $k$-dimensional graph.
%%%%%%%%%%%%%%%%%%%%%%%%%%%%%%%%%%%%%%%%%%%%%%%%%%%%%%%%%%%%%%%%%%%%%%%%%%%%%%%%%%%%
%%%%%%%%%%%%%%%%%%%%%%%%%%%%%%%%%%%%%%%%%%%%%%%%%%%%%%%%%%%%%%%%%%%%%%%%%%%%%%%%%%%%%%%
\section{Preliminaries}
In this section, we present some definitions and known results
which are necessary  to prove our main theorems. Throughout this
paper, $G=(V,E)$ is a finite, simple, and connected  graph with
$e(G)$ edges. The distance between two vertices $u$ and $v$,
denoted by $d(u,v)$, is the length of a shortest path between $u$
and $v$ in $G$. The {\it eccentricity} of a vertex $v\in V(G)$ is
$e(v)=\max_{u\in V(G)}d(u,v)$ and the {\it diameter} of $G$ is
$\max_{v\in V(G)}e(v)$. We use $\Gamma_i(v)$ for the set of all
vertices $u\in V(G)$ with $d(u,v)=i$. Also, $N_G(v)$ is the set of
all neighbors of vertex $v$ in $G$ and $\deg_{_G}(v)=|N_G(v)|$ is the
{\it degree} of vertex $v$. For a set $S\subseteq V(G)$,
$N_G(S)=\bigcup_{v\in S}N_G(v)$. If $G$ is clear from the context, it
is customary to write $N(v)$ and $\deg(v)$ rather than $N_G(v)$ and
$\deg_{_G}(v)$, respectively. The {\it maximum degree} and {\it
minimum degree} of $G$, are denoted by $\Delta(G)$ and
$\delta(G)$, respectively. For a subset $S$
of $V(G)$, $G\setminus S$ is the induced subgraph $\langle
V(G)\setminus S\rangle$ of $G$. A set
$S\subseteq V(G)$ is a {\it separating set} in $G$ if $G\setminus
S$ has at least two components. Also, a set $T\subseteq E(G)$ is
an {\it edge cut} in $G$ if $G\setminus T$ has at least two
components. A graph $G$ is $k$-(edge-)connected if the minimum
size of a separating set (edge cut) in $G$ is at least $k$. We mean by
$\omega(G)$, the number of vertices in a maximum clique in $G$.
The notations $u\sim v$ and $u\nsim v$ denote the adjacency and
non-adjacency relations between $u$ and $v$, respectively. The
symbols $(v_1,v_2,\ldots, v_n)$ and $(v_1,v_2,\ldots,v_n,v_1)$
represent a path of order $n$, $P_n$, and a cycle of order $n$,
$C_n$, respectively.
\par For an ordered set
$W=\{w_1,w_2,\ldots,w_k\}\subseteq V(G)$ and a vertex $v$ of $G$,
the  $k$-vector
$$r(v|W):=(d(v,w_1),d(v,w_2),\ldots,d(v,w_k))$$
is called  the ({\it metric}) {\it representation}  of $v$ with
respect to $W$. The set $W$ is called a {\it resolving set} for
$G$ if distinct vertices have different representations. In this
case, we say  set $W$ resolves $G$. To see that whether a given
set $W$ is a resolving set for $G$, it is sufficient to look at
the representations of vertices in $V(G)\backslash W$, because
$w\in W$ is the unique vertex of $G$ for which $d(w,w)=0$. A
resolving set $W$ for $G$ with minimum cardinality is  called a
{\it basis} of $G$, and its cardinality is the {\it metric
dimension} of $G$, denoted by $\beta(G)$.  The concept of (metric)
representation is introduced by Slater~\cite{Slater1975}
(see~\cite{Harary}). For more results related to these concepts
see~\cite{trees,cartesian product,bounds,sur1,cayley
digraphs,landmarks,sur2}.
\par We say an ordered set $W$ {\it resolves} a set $T$ of vertices
in $G$, if the representations of vertices in $T$ are distinct
with respect to $W$. When $W=\{x\}$, we say that  vertex $x$
resolves $T$.  The following simple result is very useful.
\begin{obs}~{\rm\cite{extermal}}\label{twins}
Suppose that $u,v$ are vertices in $G$ such that
$N(v)\backslash\{u\}=N(u)\backslash\{v\}$ and $W$ resolves $G$.
Then $u$ or $v$ is in $W$. Moreover, if $u\in W$ and $v\notin W$,
then $(W\setminus\{u\})\cup \{v\}$ also resolves $G$.
\end{obs}
Let $G$ be a graph of order $n$. It is obvious that
$1\leq\beta(G)\leq n-1$. The following theorem characterize all
graphs $G$ with $\beta(G)=1$ and $\beta(G)=n-1$.
\begin{lem}~\rm\cite{Ollerman}\label{B=1,B=n-1}
Let $G$ be a graph of order $n$. Then,\begin{description}\item (i)
$\beta(G)=1$ if and only if $G=P_n$,\item (ii) $\beta(G)=n-1$ if
and only if $G=K_n$.
\end{description}
\end{lem}
\par The {\it basis number}  of $G$, $bas(G)$, is the largest integer $r$
such that every $r$-set of vertices of $G$ is a subset of some
basis of $G$. Also, the {\it resolving number}  of $G$, $res(G)$,
is the minimum $k$ such that every $k$-set of vertices of $G$ is a
resolving set for $G$. These parameters are introduced
in~\cite{basis} and ~\cite{res(G)}, respectively. Clearly, if $G$
is a graph of order $n$, then $0\leq bas(G)\leq \beta(G)$ and
$\beta(G)\leq res(G)\leq n-1$. Chartrand et al.~\cite{basis}
considered graphs $G$ with $bas(G)=\beta(G)$. They called these
graphs {\it randomly $k$-dimensional}, where $k=\beta(G)$. Obviously,
$bas(G)=\beta(G)$ if and only if $res(G)=\beta(G)$. In other word,
a graph $G$ is randomly $k$-dimensional if each $k$-set of
vertices of $G$ is a basis of $G$.
\par The following properties of randomly $k$-dimensional
graphs are proved in~\cite{On randomly K-dimensional}.
\begin{prop}\label{not twin}~\rm\cite{On randomly K-dimensional}
If $G\neq K_n$ is a randomly
$k$-dimensional graph, then  for each pair of vertices $u,v\in
V(G)$, $N(v)\backslash\{u\}\neq N(u)\backslash\{v\}$.
\end{prop}
\begin{lem}\label{2connected}~\rm\cite{On randomly K-dimensional}
 If $k\geq2$, then every randomly $k$-dimensional graph
is $2$-connected.
\end{lem}
\begin{lem}\label{|T|=k-1}~\rm\cite{On randomly K-dimensional}
If $G$ is a randomly $k$-dimensional graph and $T$ is a separating
set of $G$ with $|T|=k-1$, then $G\setminus T$ has exactly two
components. Moreover, for each pair of vertices $u,v\in V(G)\setminus T$
with $r(u|T)=r(v|T)$, $u$ and $v$ belong to different components.
\end{lem}
\begin{lem}\label{at most k-1 common neighbor}~\rm\cite{On randomly K-dimensional}
 If $res(G)=k$,
 then each two vertices of $G$
 have at most $k-1$ common neighbors.
\end{lem}
Chartrand et al. in \cite{basis} characterized the randomly
$2$-dimensional graphs and prove that a graph $G$ is randomly
$2$-dimensional if and only if $G$ is an odd cycle. Furthermore,
they  provided the following question.
\begin{ques}~\rm\cite{basis}\label{question}  Are there randomly $k$-dimensional graphs
other than complete graph and odd cycles?
\end{ques}
In  this paper, we prove that the answer of
Question~\ref{question} is negative and a graph $G$ is a randomly
$k$-dimensional graph with $k\geq3$ if and only if $G=K_{k+1}$.
%%%%%%%%%%%%%%%%%%%%%%%%%%%%%%%%%%%%%%%%%%%%%%%%%%%%%%%%%%%%%%%%%%%%%%%%%%%%%%%%%%%%%%%%%%%%%%%%%%%%%%%%%
%%%%%%%%%%%%%%%%%%%%%%%%%%%%%%%%%%%%%%%%%%%%%%%%%%%%%%%%%%%%%%%%%%%%%%%%%%%%%%%%%%%%%%%%%%%%%%%%%%%%%%%%%%%%%%
%%%%%%%%%%%%%%%%%%%%%%%%%%%%%%%%%%%%%%%%%%%%%%%%%%%%%%%%%%%%%%%%%%%%%%%%%%%%%%%%%%%%%%%%%%%%%%%%%%%%%%%%
\section{Some Properties of Randomly $\bf k$-Dimensional Graphs}\label{main}
Let $V_p$ denote the collection of all $n \choose 2$ pairs of
vertices of $G$. Fehr et al.~\cite{cayley digraphs} defined the
{\it resolving graph} $R(G)$  of $G$ as a bipartite graph with
bipartition $(V(G),V_p)$, where a vertex $v\in V(G)$ is adjacent
to a pair $\{x,y\}\in V_p$ if and only if $v$ resolves $\{x,y\}$
in $G$. Thus, the minimum cardinality of a subset $S$ of $V(G)$,
where $N_{_{R(G)}}(S)=V_p$  is the metric dimension of $G$.
\par In the following through some propositions and lemmas, we prove that
if $G$ is a randomly $k$-dimensional graph of order $n$ and
diameter $d$, then $k\geq {n-1\over d}$.
\begin{pro}\label{bounds on edgs of R(G)}
If $G$ is a randomly $k$-dimensional graph of order $n$, then
$${n\choose2}(n-k+1)\leq e(R(G))\leq n({n\choose2}-k+1).$$
\end{pro}
\begin{proof}{Let $z\in V_p$ and $S=\{v\in V(G)\,|\,v\nsim z\}$.
Thus, $N_{R(G)}(S)\neq V_p$ and hence, $S$ is not a resolving set
for $G$. If $\deg_{_{R(G)}} (z)\leq n-k$, then $|S|\geq k$, which
contradicts $res(G)=k$. Therefore, $\deg_{_{R(G)}} (z)\geq n-k+1$
and consequently, $ e(R(G))\geq {n\choose2}(n-k+1)$.
\par Now, let $v\in V(G)$. If $\deg_{_{R(G)}}(v)\geq
{n\choose2}-k+2$, then there are at most $k-2$ vertices in $V_p$
which are not adjacent to $v$. Let $V_p\setminus
N_{R(G)}(v)=\{\{u_1,v_1\}, \{u_2,v_2\},\ldots,\{u_t,v_t\}\}$,
where $t\leq k-2$. Note that, $u_i\sim\{u_i,v_i\}$ in $R(G)$ for
each $i$, $1\leq i\leq t$. Therefore,
$N_{R(G)}(\{v,u_1,u_2,\ldots,u_t\})=V_p$. Hence, $\beta(G)\leq
t+1\leq k-1$, which is a contradiction. Thus,
$\deg_{_{R(G)}}(v)\leq {n\choose2}-k+1$ and consequently,
$e(R(G))\leq n({n\choose2}-k+1).$ }\end{proof}
\begin{pro}\label{edgs of R(G)} If $G$ is a randomly
$k$-dimensional graph of order $n$, then for each $v\in V(G)$,
$$\deg_{_{R(G)}}(v)={n\choose2}-\sum_{i=1}^{e(v)}{|\Gamma_i(v)|\choose2}.$$
\end{pro}
\begin{proof}{Note that, a vertex $v\in V(G)$ resolves a pair
$\{x,y\}$ if and only if there exist  $0\leq i\neq j\leq e(v)$
such that $x\in \Gamma_i(v)$ and $y\in \Gamma_j(v)$. Therefore, a
vertex $\{u,w\}\in V_p$ is not adjacent to $v$ in $R(G)$ if and
only if there exists an $i,~1\leq i\leq e(v)$, such that $u,w\in
\Gamma_i(v)$. The number of such vertices in $V_p$ is
$\sum_{i=1}^{e(v)}{|\Gamma_i(v)|\choose2}.$ Therefore,
$\deg_{_{R(G)}}(v)={n\choose2}-\sum_{i=1}^{e(v)}{|\Gamma_i(v)|\choose2}.$
}\end{proof} Since $R(G)$ is bipartite, by Proposition~\ref{edgs of
R(G)},
$$e(R(G))=\sum_{v\in V(G)}[{n\choose2}-\sum_{i=1}^{e(v)}{|\Gamma_i(v)|\choose2}]\\
=n{n\choose2}-\sum_{v\in
V(G)}\sum_{i=1}^{e(v)}{|\Gamma_i(v)|\choose2}.$$ Thus, by
Proposition~\ref{bounds on edgs of R(G)},
\begin{equation}\label{bounnd}
n(k-1)\leq\sum_{v\in
V(G)}\sum_{i=1}^{e(v)}{|\Gamma_i(v)|\choose2}\leq
{n\choose2}(k-1).
\end{equation}
\begin{obs}\label{min chose} Let $n_1,...,n_r$ and $n$ be positive
integers, with $\sum_{i=1}^rn_i=n$. Then,
$\sum_{i=1}^r{n_i\choose2}$ is minimum if and only if
$|n_i-n_j|\leq1$, for each  $1\leq i,j\leq r$.
\end{obs}
\begin{lemma}\label{P1<p2} Let $n,p_1,p_2,q_1,q_2,r_1$ and $r_2$ be
positive integers, such that $n=p_iq_i+r_i$ and $r_i<p_i$, for
$1\leq i\leq 2$. If $p_1<p_2$, then
$$(p_1-r_1){q_1\choose2}+r_1{q_1+1\choose2}\geq
(p_2-r_2){q_2\choose2}+r_2{q_2+1\choose2}.$$
\end{lemma}
\begin{proof}{ Let
$f(p_i)=(p_i-r_i){q_i\choose2}+r_i{q_i+1\choose2}$, $1\leq
i\leq2$. We just need to prove that $f(p_1)\geq f(p_2)$.
\begin{eqnarray*}
f(p_1)-f(p_2)
&=&{1\over2}[(p_1-r_1)q_1(q_1-1)+r_1q_1(q_1+1)-(p_2-r_2)q_2(q_2-1)-r_2q_2(q_2+1)]\\
&=&{1\over2}q_1[p_1q_1-p_1+2r_1]-{1\over2}q_2[p_2q_2-p_2+2r_2]\\
&=&{1\over2}q_1[n-p_1+r_1]-{1\over2}q_2[n-p_2+r_2]\\
&=&{1\over2}[n(q_1-q_2)-p_1q_1+r_1q_1+p_2q_2-r_2q_2].
\end{eqnarray*}
Since $p_1<p_2$, we have $q_2\leq q_1$. If $q_1=q_2$, then
$r_2<r_1$. Therefore,
$$f(p_1)-f(p_2)={1\over2}q_1[(p_2-p_1)+(r_1-r_2)]\geq0.$$
If $q_2<q_1$, then $q_1-q_2\geq1$. Thus,
$$f(p_1)-f(p_2)\geq{1\over2}[n-p_1q_1+r_1q_1+q_2(p_2-r_2)]=
{1\over2}[r_1+r_1q_1+q_2(p_2-r_2)]\geq0.$$
 }\end{proof}
\begin{theorem}\label{k>(n-1)/d}
If $G$ is a randomly $k$-dimensional graph of order $n$ and
diameter $d$, then $\displaystyle k\geq {n-1\over d}$.
\end{theorem}
\begin{proof}{ Note that, for each $v\in V(G)$,
$|\bigcup_{i=1}^{e(v)}\Gamma_i(v)|=n-1$. For $v\in V(G)$, let
$n-1=q(v)e(v)+r(v)$,  where $0\leq r(v)<e(v)$. Then, by
Observation~\ref{min chose},
\begin{equation}\label{2}(e(v)-r(v)){q(v)\choose2}+r(v){q(v)+1\choose2}\leq
\sum_{i=1}^{e(v)}{|\Gamma_i(v)|\choose2}.
\end{equation}
 %Also, $e(v)\leq e(u)$
%implies that $q(u)\geq q(v)$ and hence,
%$$(e(u)-r(u)){q(u)\choose2}+r(u){q(u)+1\choose2}\leq
%(e(v)-r(v)){q(v)\choose2}+r(v){q(v)+1\choose2}.$$
Let $w\in V(G)$ with $e(w)=d,~r(w)=r$, and $q(w)=q$, then
$n-1=qd+r$. Since for each $v\in V(G)$, $e(v)\leq e(w)$, by
Lemma~\ref{P1<p2},
$$(d-r){q\choose2}+r{q+1\choose2}\leq
(e(v)-r(v)){q(v)\choose2}+r(v){q(v)+1\choose2}.$$  Therefore,
$$n[(d-r){q\choose2}+r{q+1\choose2}]\leq\sum_{v\in V(G)}[
(e(v)-r(v)){q(v)\choose2}+r(v){q(v)+1\choose2}].$$ Thus, by
Relations~(\ref{2}) and (\ref{bounnd}),
$$n[(d-r){q\choose2}+r{q+1\choose2}]\leq\sum_{v\in V(G)}
\sum_{i=1}^{e(v)}{|\Gamma_i(v)|\choose2}\leq{n\choose2}(k-1).$$
Hence, $q[(d-r)(q-1)+r(q+1)]\leq (n-1)(k-1)$, which implies,
$q[(r-d)+(d-r)q+r(q+1)]\leq (n-1)(k-1)$. Therefore,
$q(r-d)+q(n-1)\leq (n-1)(k-1)$. Since $q=\lfloor{{n-1}\over
d}\rfloor$, we have
\begin{eqnarray*}
k-1\geq q+q{r-d\over n-1}&=&q+{qr\over n-1}-{qd\over n-1}\\
&=&q+{qr\over n-1}-{\lfloor{{n-1}\over d}\rfloor d\over n-1}\\
&\geq&q+{qr\over n-1}-1.
\end{eqnarray*} Thus, $k\geq \lfloor{{n-1}\over
d}\rfloor+{qr\over n-1}$. Note that, ${qr\over n-1}\geq0$. If
${qr\over n-1}>0$, then $k\geq \lceil{{n-1}\over d}\rceil$, since
$k$ is an integer. If ${qr\over n-1}=0$, then $r=0$ and
consequently, $d$ divides $n-1$. Thus, $\lfloor{{n-1}\over
d}\rfloor=\lceil{{n-1}\over d}\rceil$. Therefore, $k\geq
\lceil{{n-1}\over d}\rceil\geq {{n-1}\over d}\,\cdot$ }\end{proof}
The following theorem shows that there is no randomly
$k$-dimensional graph of order $n$, where $4\leq k\leq n-2$.
\begin{theorem}\label{bonds on k}
If $G$ is a randomly $k$-dimensional graph of order $n$, then
$k\leq3$ or $k\geq n-1$.
\end{theorem}
\begin{proof}{ For each $W\subseteq V(G)$, let $\overline N(W)=V_p\setminus N(W)$ in $R(G)$.
We claim that, if $S,T\subseteq V(G)$ with $|S|=|T|=k-1$ and
$T\neq S$, then $\overline N(S)\cap \overline N(T)=\emptyset$.
Otherwise, there exists a pair $\{x,y\}\in \overline N(S)\cap
\overline N(T)$. Therefore, $\{x,y\}\notin N(S\cup T)$ and hence,
$S\cup T$ is not a resolving set for $G$. Since $S\neq T$, $|S\cup
T|>|S|=k-1$, which contradicts $res(G)=k$. Thus, $\overline
N(S)\cap \overline N(T)=\emptyset$.
\par Since $\beta(G)=k$,
for each $S\subseteq V(G)$ with $|S|=k-1$, $\overline
N(S)\neq\emptyset$. Now, let $\Omega=\{S\subseteq
V(G)\,|\,|S|=k-1\}$. Therefore,
$$|\bigcup_{S\in\Omega}\overline
N(S)|=\sum_{S\in\Omega}|\overline N(S)|\geq
\sum_{S\in\Omega}1={n\choose k-1}.$$ On the other hand,
$\bigcup_{S\in\Omega}\overline N(S)\subseteq V_p$. Hence,
$|\bigcup_{S\in\Omega}\overline N(S)|\leq {n\choose2}$.
Consequently, ${n\choose k-1}\leq{n\choose2}$. If $n\leq4$, then
$k\leq3$. Now, let $n\geq5$. Thus, $2\leq {n+1\over2}$. We know
that for each $a,b\leq{n+1\over2}$, ${n\choose a}\leq {n\choose
b}$ if and only if $a\leq b$. Therefore, if $k-1\leq{n+1\over2}$,
then $k-1\leq 2$, which implies $k\leq3$. If $k-1\geq{n+1\over2}$,
then $n-k+1\leq{n+1\over2}$. Since ${n\choose n-k+1}={n\choose
k-1}$, we have ${n\choose n-k+1}\leq {n\choose 2}$ and
consequently, $n-k+1\leq2$, which yields $k\geq n-1$. }\end{proof}
By Theorem~\ref{bonds on k}, to characterize all randomly
$k$-dimensional graphs, we only need to consider graphs of order $k+1$
and graphs with metric dimension less than four. By
Theorem~\ref{B=1,B=n-1}, if $G$ has $k+1$ vertices and
$\beta(G)=k$, then $G=K_{k+1}$. Also, if $k=1$, then $G=P_n$.
Clearly, the only paths with resolving number $1$ are $P_1=K_1$
and $P_2=K_2$. Furthermore, randomly $2$-dimensional graphs are determined
in~\cite{basis} and it has been proved that these graphs are odd cycles. Therefore, to complete the
characterization, we only need to investigate randomly
$3$-dimensional graphs.
%%%%%%%%%%%%%%%%%%%%%%%%%%%%%%%%%%%%%%%%%%%%%%%%%%%%%%%%%%%%%%%%%%%%%%%%%%%%%%%%%%%%%%%%%%%%%%%%%%%%%%%%%
%%%%%%%%%%%%%%%%%%%%%%%%%%%%%%%%%%%%%%%%%%%%%%%%%%%%%%%%%%%%%%%%%%%%%%%%%%%%%%%%%%%%%%%%%%%%%%%%%%%%%%%%%%%%%%
%%%%%%%%%%%%%%%%%%%%%%%%%%%%%%%%%%%%%%%%%%%%%%%%%%%%%%%%%%%%%%%%%%%%%%%%%%%%%%%%%%%%%%%%%%%%%%%%%%%%%%%%
\section{Randomly $\bf3$-Dimensional Graphs}\label{3dimensional}
In this section, through several lemmas and theorems, we prove
that the complete graph $K_4$ is the unique randomly
$3$-dimensional graph.
\begin{pro}\label{DELTA} If $res(G)=k$, then $\Delta(G)\leq
2^{k-1}+k-1$.
\end{pro}
\begin{proof}{ Let $v\in V(G)$ be a vertex with
$\deg(v)=\Delta(G)$ and $T=\{v,v_1,v_2,\ldots,v_{k-1}\}$, where
$v_1,v_2,\ldots,v_{k-1}$ are neighbors of $v$. Since $res(G)=k$,
$T$ is a resolving set for $G$. Note that, $d(u,v)=1$ and
$d(u,v_i)\in\{1,2\}$ for each $u\in N(v)\setminus T$ and each
$i,~1\leq i\leq k-1$. Therefore, the maximum number of distinct
representations for vertices of $N(v)\setminus T$ is  $2^{k-1}$.
Since $T$ is a resolving set for $G$, the representations of
vertices of $N(v)\setminus T$ are distinct. Thus, $|N(v)\setminus
T|\leq 2^{k-1}$ and hence, $\Delta(G)=|N(v)|\leq 2^{k-1}+k-1$.
}\end{proof}
\begin{lemma}\label{Delta<6} If $res(G)=3$, then $\Delta(G)\leq5$.
\end{lemma}
\begin{proof}{By Proposition~\ref{DELTA},
$\Delta(G)\leq6$. Suppose on the contrary that, there exists a
vertex $v\in V(G)$ with $\deg(v)=6$ and
$N(v)=\{x,y,v_1,\ldots,v_4\}$. Since $res(G)=3$, set $\{v,x,y\}$
is a resolving set for $G$. Therefore, the representations of
vertices $v_1,\ldots,v_4$ with respect to this set are
$r_1=(1,1,1)$, $r_2=(1,1,2)$, $r_3=(1,2,1)$, and $r_4=(1,2,2)$.
Without loss of generality, we can assume $r(v_i|\{v,x,y\})=r_i$,
for each $i,~1\leq i\leq4$. Thus, $y\nsim v_2,~y\nsim v_4$, and
$y\sim v_3$.
\par On the other hand, set $\{v,y,v_3\}$ is a
resolving set for $G$, too. Hence, the representations of vertices
$x,v_1,v_2,v_4$ with respect to this set are $r_1,r_2,r_3,r_4$ in
some order. Therefore, the vertex $y$ has two neighbors and two
non-neighbors in $\{x,v_1,v_2,v_4\}$. Since $y\nsim v_2$ and
$y\nsim v_4$, the vertices $x,v_1$ are adjacent to $y$. Thus,
$r(y|\{x,v_1,v_3\})=(1,1,1)=r(v|\{x,v_1,v_3\})$, which contradicts
$res(G)=3$. Hence, $\Delta(G)\leq 5$. }\end{proof}
\begin{lemma}\label{induced=C5}
If $res(G)=3$ and $v\in V(G)$ is a vertex with $\deg(v)=5$, then
the induced subgraph $\langle N(v)\rangle$ is a cycle $C_5$.
\end{lemma}
\begin{proof}{Let $H=\langle N(v)\rangle$.
By Theorem~\ref{at most k-1 common neighbor}, for each $x\in N(v)$
we have, $|N(x)\cap N(v)|\leq2$. Therefore, $\Delta(H)\leq2$, thus,
each component of $H$ is a path or a cycle. If
the largest component of $H$ has at most three vertices, then
there are two vertices $x,y\in N(v)$ which are not adjacent to any
vertex in $N(v)\setminus\{x,y\}$. Thus, for each $u\in
N(v)\setminus\{x,y\}$, $r(u|\{v,x,y\})=(1,2,2)$, which contradicts
$res(G)=3$. Therefore, the largest component of $H$, say $H_1$,
has at least four vertices and the other component has at most one
vertex, say $\{x\}$. Let $(y_1,y_2,y_3)$ be a path in $H_1$. Hence $r(y_1|\{v,x,y_2\})=(1,2,1)=r(y_3|\{v,x,y_2\})$, which is
a contradiction. Therefore, $H=C_5$ or $H=P_5$. If
$H=P_5=(y_1,y_2,y_3,y_4,y_5)$, then
$r(y_4|\{v,y_1,y_2\})=(1,2,2)=r(y_5|\{v,y_1,y_2\})$, which is
impossible. Therefore, $H=C_5$. }\end{proof}
\begin{lemma}\label{induced=P4}
If $res(G)=3$ and $v\in V(G)$ is a vertex with $\deg(v)=4$, then
the induced subgraph $\langle N(v)\rangle$ is a path $P_4$.
\end{lemma}
\begin{proof}{Let $H=\langle N(v)\rangle$.
By Theorem~\ref{at most k-1 common neighbor}, for each $x\in N(v)$,
we have $|N(x)\cap N(v)|\leq2$. Hence, $\Delta(H)\leq2$ thus,
each component of $H$ is a path or a cycle.  If $H$
has more than two components, then it has at least two components
with one vertex say $\{x\}$ and $\{y\}$. Thus,
$r(u|\{v,x,y\})=(1,2,2)$, for each $u\in N(v)\setminus\{x,y\}$,
which contradicts $res(G)=3$. If $H$ has exactly two  components
$H_1=\{x,y\}$ and $H_2=\{u,w\}$, then
$r(u|\{v,x,y\})=(1,2,2)=r(w|\{v,x,y\})$, which is a contradiction.
Now, let $H$ has a component with one vertex, say $\{x\}$, and a
component contains a path $(y_1,y_2,y_3)$. Consequently,
$r(u|\{v,x,y_2\})=(1,2,1)$, for each $u\in
N(v)\setminus\{x,y\}$, which is a  contradiction. Therefore, $H=C_4$ or
$H=P_4$. If $H=C_4=(y_1,y_2,y_3,y_4,y_1)$, then
$r(y_1|\{v,y_2,y_4\})=(1,1,1)=r(y_3|\{v,y_2,y_4\})$, which is
impossible. Therefore, $H=P_4$. }\end{proof}
\begin{pro}\label{k=3 implies Delta<3}
If $G$ is a randomly $3$-dimensional graph, then $\Delta(G)\leq3$.
\end{pro}
\begin{proof}{ By Lemma~\ref{Delta<6},
$\Delta(G)\leq5$.  If there exists a vertex $v\in V(G)$ with
$\deg(v)=5$, then, by Lemma~\ref{induced=C5}, $\langle
N(v)\rangle=C_5$. If $\Gamma_2(v)=\emptyset$, then $G=C_5\vee K_1$
(the join of graphs $C_5$ and $K_1$) and hence, $\beta(G)=2$,
which is a contradiction. Thus, $\Gamma_2(v)\neq\emptyset$. Let
$u\in \Gamma_2(v)$. Then $u$ has a neighbor in $N(v)$, say $x$.
Since $\langle N(v)\rangle=C_5$, $x$ has exactly two neighbors in
$N(v)$, say $x_1,x_2$. Therefore, $\deg(x)\geq4$. By
Lemmas~\ref{induced=C5} and \ref{induced=P4},
$\langle\{u,v,x_1,x_2\}\rangle=P_4$. Note that, by Theorem~\ref{at
most k-1 common neighbor}, $u$ has at most two neighbors in
$N(v)$. Thus, $u$ is adjacent to exactly one of $x_1$ and $x_2$,
say $x_1$.  As in Figure~\ref{12}(a), the set $\{u,v,s\}$ is not a
resolving set for $G$, because
$r(x|\{u,v,s\})=(1,1,2)=r(x_1|\{u,v,s\})$. This contradiction
implies that $\Delta(G)\leq4$.
\par If $v$ is a vertex of degree four in $G$, then by
Lemma~\ref{induced=P4}, $\langle N(v)\rangle=P_4$. Let $\langle
N(v)\rangle=(x_1,x_2,x_3,x_4)$. If $\Gamma_2(v)=\emptyset$, then
$G=P_4\vee K_1$ and consequently, $\beta(G)=2$, which is a
contradiction. Thus, $\Gamma_2(v)\neq\emptyset$. Let
$u\in\Gamma_2(v)$. Then, $u$ has a neighbor in $N(v)$ and by
Theorem~\ref{at most k-1 common neighbor}, $u$ has at most two
neighbors in $N(v)$. If $u$ has only one neighbor in $N(v)$, then
by symmetry, we can assume $u\sim x_1$ or $u\sim x_2$. If $u\sim
x_2$ and $u\nsim
x_1$, then $\deg(x_2)=4$ and by Lemma~\ref{induced=P4}, $\langle
\{u,x_1,x_3,v\}\rangle=P_4$. Therefore, $u$ has two neighbors in
$N(v)$, which is a contradiction. If $u\sim x_1$ and $u\nsim
x_2$, then
$r(v|\{x_1,x_3,u\})=(1,1,2)=r(x_2|\{x_1,x_3,u\})$, which
contradicts $res(G)=3$. Hence, $u$ has exactly two neighbors in
$N(v)$. Let $T=N(u)\cap N(v)$. By symmetry, we can assume that $T$
is one of the sets $\{x_1,x_2\},\{x_1,x_3\},\{x_1,x_4\}$, and
$\{x_2,x_3\}$. If $T=\{x_1,x_2\}$, then
$r(x_1|\{v,x_4,u\})=(1,2,1)=r(x_2|\{v,x_4,u\})$. If
$T=\{x_1,x_3\}$, then
$r(x_1|\{v,x_2,u\})=(1,1,1)=r(x_3|\{v,x_2,u\})$. If
$T=\{x_1,x_4\}$, then
$r(v|\{x_1,x_3,u\})=(1,1,2)=r(x_2|\{x_1,x_3,u\})$. These
contradictions, imply that $T=\{x_2,x_3\}$. Thus,
$|\Gamma_2(v)|=1$, because each vertex of $\Gamma_2(v)$ is
adjacent to both vertices $x_2$ and $x_3$ and if $\Gamma_2(v)$ has
more than one vertex, then $\deg(x_2)=\deg(x_3)\geq5$, which is
impossible. Now, if $\Gamma_3(v)=\emptyset$, then $\{x_1,x_4\}$ is
a resolving set for $G$, which is a contradiction. Therefore,
$\Gamma_3(v)\neq\emptyset$ and hence, $u$ is a cut vertex in $G$,
which contradicts the $2$-connectivity of $G$
(Theorem~\ref{2connected}). Consequently, $\Delta(G)\leq3$.
}\end{proof}
%\begin{figure}[h]
%\centering
%\unitlength=.6mm
%%\begin{picture}(90,48)
%\epsfysize=5cm
%\epsfbox[0 30 300 130]{graph fig1.eps}
%%\begin{figure}[h]
%%\unitlength=1.1mm
%%\epsfxsize=45mm
%%\epsfysize=10cm
%%\hspace{6cm}{\epsffile{n-3.eps}}
%%%\caption{\mbox{\InF{}  ð‰Âé \EnF{}$G^*$\InF{}  ¢¤  ì‰Ì‰ƒ‰'ý $4$ \EnF{}}}
%%\end{figure}

\vspace{-.7cm}\begin{figure}[ht]
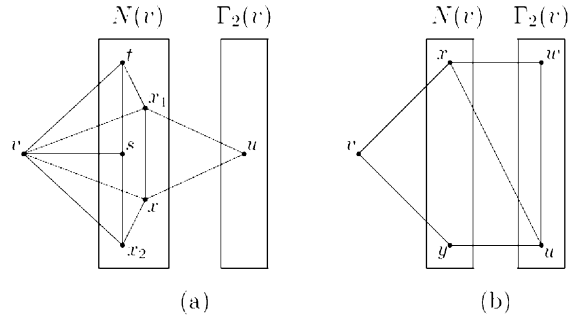
\hspace{3cm}
%\unitlength=1.1mm
\vspace*{5cm}\special{em:graph fig1.bmp} \vspace*{1.5cm}
\vspace{-1.3cm}\caption{ $(a)$ $\Delta(G)=5$, $(b)$ Neighbors of a
vertex of degree $2$.\label{12}}
\end{figure}

%\caption{\label{12} $(a)$ $\Delta(G)=5$, $(b)$ Neighbors of a
%vertex of degree $2$.}
%\end{figure}
\begin{theorem}\label{3-regular}
If $G$ is a randomly $3$-dimensional graph, then $G$ is
$3$-regular.
\end{theorem}
\begin{proof}{By Proposition~\ref{k=3 implies Delta<3},
$\Delta(G)\leq3$ and by Theorem~\ref{2connected},
$\delta(G)\geq2$. Suppose that, $v$ is a vertex of degree $2$ in
$G$. Let $N(v)=\{x,y\}$. Since $N(v)$ is a separating set of size
$2$ in $G$, Theorem~\ref{|T|=k-1} implies that
$G\setminus\{v,x,y\}$ is a connected graph and there exists a
vertex $u\in V(G)\setminus\{v,x,y\}$ such that $u\sim x$ and
$u\sim y$. Note that $G\neq K_n$, because $G$ has a vertex of
degree $2$ and $\beta(G)=3$. Thus, by Proposition~\ref{not twin},
there exists a vertex $w\in V(G)$ such that $w\sim u$ and $w\nsim
v$.
\par If $w$ is neither  adjacent to $x$ nor $y$, then
$r(x|\{v,u,w\})=(1,1,2)=r(y|\{v,u,w\})$, which contradicts
$res(G)=3$. Also, if $w$ is adjacent to both of vertices $x$ and
$y$, then $r(x|\{v,u,w\})=(1,1,1)=r(y|\{v,u,w\})$, which is a
contradiction. Hence, $w$ is adjacent to exactly one of the
vertices $x$ and $y$, say $x$. Since $\Delta(G)\leq3$, the graph
in Figure~\ref{12}(b) is an induced subgraph of $G$. Clearly, the
metric dimension of this subgraph is $2$. Therefore, $G$ has at least six
vertices.
\par If $|\Gamma_2(v)|=2$, then $w$ is a cut vertex in $G$,
because $\Delta(G)\leq3$. This contradiction implies that there
exists a vertex $z$ in $\Gamma_2(v)\setminus\{u,w\}$. Since
$\Delta(G)\leq3$, $z\sim y$. If $z\sim w$, then the graph in
Figure~\ref{34}(a) is an induced subgraph of $G$ which its metric dimension
is $2$. In this case, $G$ must has at least seven vertices and
consequently, $z$ is a cut vertex in $G$, which contradicts
Theorem~\ref{2connected}. Hence, $z\nsim w$. By
Theorem~\ref{2connected}, $\deg(z)\geq2$. Therefore, $z$ has a
neighbor in $\Gamma_3(v)$. If there exists a vertex $s\in
\Gamma_3(v)$ such that $s\sim z$ and $s\nsim w$, then
$r(v|\{y,z,s\})=(1,2,3)=r(u|\{y,z,s\})$, which contradicts
$res(G)=3$. Thus, $w$ is adjacent to all neighbors of $z$ in
$\Gamma_3(v)$. Since
$\Delta(G)\leq 3$, $z$ has exactly one neighbor in $\Gamma_3(v)$,
say $t$. Hence $\Gamma_3(v)=\{t\}$.
\par If $G$ has more vertices, then $t$ is a cut vertex in $G$,
which contradicts the $2$-connectivity of $G$. Therefore, $G$ is as
Figure~\ref{34}(b) and consequently, $\beta(G)=2$, which is a
contradiction. Thus, $G$ does not have any vertex of degree $2$.
}\end{proof}
\vspace{-.5cm}\begin{figure}[ht]
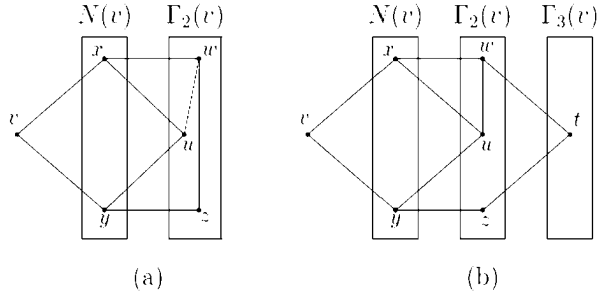
\hspace{2cm}
\vspace*{5cm}\special{em:graph fig2.bmp} \vspace*{1.5cm}
%\caption{ $(a)$ $\Delta(G)=5$, $(b)$ Neighbors of a
%vertex of degree $2$.\label{12}}
%\end{figure}
\vspace{-2cm}\caption{\label{34} The minimum degree of $G$ is more than $2$.}
\end{figure}
%\begin{lem}\label{kapa=kapa'}~\rm\cite{west} If $G$ is a
%$3$-regular graph, then  connectivity and edge-connectivity of $G$
%are equal.
%\end{lem}
\begin{theorem}\label{3-connected}
If $G$ is a randomly $3$-dimensional graph, then $G$ is
$3$-connected.
\end{theorem}
\begin{proof}{Suppose on the contrary that, $G$ is not $3$-connected.
Therefore, by Theorem~\ref{2connected}, the connectivity of $G$ is
$2$. Since $G$ is $3$-regular, (by Theorem~$4.1.11$ in~\cite{west},) the
edge-connectivity of $G$ is also $2$. Thus, there exists a minimum
edge cut in $G$ of size $2$, say $\{xu,yv\}$. Let $H$ and $H_1$ be
components of $G\setminus\{xu,yv\}$ such that $x,y\in V(H)$ and $u,v\in
V(H_1)$. Note that, $x\neq y$ and $u\neq v$, because $G$ is
$2$-connected. Since $G$ is $3$-regular, $|H|\geq3$ and
$|H_1|\geq3$. Therefore, $\{x,y\}$ is a separating set in $G$ and
components of $G\setminus\{x,y\}$ are $H_1$ and
$H_2=H\setminus\{x,y\}$. Hence, each of vertices $x$ and $y$ has
exactly one neighbor in $H_1$, $u$ and $v$, respectively. Since
$G$ is $3$-regular, $x$ has at most two neighbors in $H_2$ and $u$
has exactly two neighbors $s,t$ in $H_1$. Thus, $u$ has a neighbor
in $H_1$  other than $v$, say $s$. Therefore, $s\nsim x$ and
$s\nsim y$.
\par If $x$ has two neighbors $p,q$ in $H_2$, then
$r(p|\{x,u,s\})=(1,2,3)=r(q|\{x,u,s\})$, which contradicts
$res(G)=3$. Consequently, $x$ has exactly one neighbor in $H_2$,
say $p$. Since $G$ is $3$-regular, $x\sim y$ and hence, $y$ has
exactly one neighbor in  $H_2$. Note that $p$ is not the unique
neighbor of $y$ in $H_2$, because $G$ is $2$-connected. Thus,
$d(t,p)=3$ and hence, $r(s|\{u,x,p\})=(1,2,3)=r(t|\{u,x,p\})$,
which is impossible. Therefore, $G$ is $3$-connected.
}\end{proof}
\begin{pro}\label{N(v) independent}
If $G\neq K_4$ is a randomly $3$-dimensional graph, then for each
$v\in V(G)$, $N(v)$ is an independent set in $G$.
\end{pro}
\begin{proof}{Suppose on the contrary that there exists a vertex
$v\in V(G)$, such that $N(v)$ is not an independent set in $G$. By
Theorem~\ref{3-regular}, $\deg(v)=3$.  Let $N(v)=\{u_1,u_2,u_3\}$.
Since $G\neq K_4$, the induced subgraph $\langle N(v)\rangle$ of
$G$ has one or two edges. If $\langle N(v)\rangle$  has  two
edges, then by symmetry, let $u_1\sim u_2$, $u_2\sim u_3$ and
$u_1\nsim u_3$. Since $G$ is $3$-regular, the set $\{u_1,u_3\}$ is
a separating set in $G$, which contradicts
Theorem~\ref{3-connected}. This argument implies that for
each $s\in V(G)$, $\langle N(s)\rangle$ dos not have two edges.
Hence, $\langle N(v)\rangle$  has one edge, say $u_1u_2$. Since
$G$ is $3$-regular, there are exactly four edges between $N(v)$
and $\Gamma_2(v)$. Therefore, $\Gamma_2(v)$ has at most four
vertices, because each vertex of $\Gamma_2(v)$ has a neighbor in
$N(v)$. On the other hand, $3$-regularity of $G$ forces
$\Gamma_2(v)$ has at least two vertices. Thus, one of the
following cases can happen.
\par\noindent 1. $|\Gamma_2(v)|=2$. In this case
$\Gamma_3(v)=\emptyset$, otherwise $\Gamma_2(v)$ is a separating
set of size $2$, which is impossible. Consequently, $G$ is as
Figure~\ref{56}(a). Hence, $\beta(G)=2$. But, by assumption
$\beta(G)=3$, a contradiction.
\par\noindent 2. $|\Gamma_2(v)|=3$. Let $\Gamma_2(v)=\{x,y,z\}$ and
$N(u_3)\cap\Gamma_2(v)=\{y,z\}$. Also, by symmetry, let $u_1\sim
x$, because each vertex of $\Gamma_2(v)$ has a neighbor in $N(v)$.
Then, the last edge between $N(v)$ and $\Gamma_2(v)$ is one of
$u_2x$, $u_2y$, and $u_2z$. But, $u_2x\notin E(G)$, otherwise
$\langle N(u_2)\rangle$ has two edges. Thus, by symmetry, we can
assume that $u_2y\in E(G)$ and $u_2z\notin E(G)$. Since
$res(G)=3$, we have $y\sim z$, otherwise
$r(v|\{u_2,u_3,z\})=(1,1,2)=r(y|\{u_2,u_3,z\})$, which is
impossible. For $3$-regularity of $G$, $\Gamma_3(v)\neq\emptyset$.
Hence, $\{x,z\}$ is a separating set of size $2$ in $G$, which
contradicts Theorem~\ref{3-connected}.
\par\noindent 3. $|\Gamma_2(v)|=4$. Let $\Gamma_2(v)=\{w,x,y,z\}$
and $u_1\sim w$, $u_2\sim x$, $u_3\sim y$, and $u_3\sim z$. If
$x\nsim y$ and $x\nsim z$, then $d(y,u_2)=3=d(z,u_2)$ and it
yields $r(y|\{v,u_2,u_3\})=(2,3,1)=r(z|\{v,u_2,u_3\})$. Therefore,
$G$ has at least one of the edges $xy$ and $xz$. If $G$ has both
$xy$ and $xz$, then  $r(y|\{v,x,u_3\})=r(z|\{v,x,u_3\})$. Thus,
$G$ has exactly one of the edges $xy$ and $xz$, say $xy$. On the
same way, $G$ has exactly one of the edges $wy$ and $wz$. If
$w\sim y$, then $r(x|\{v,u_3,y\})=(2,2,1)=r(w|\{v,u_3,y\})$.
Hence, $w\nsim y$ and $w\sim z$. Note that, $x\nsim w$, otherwise
$r(u_2|\{u_1,x,u_3\})=(1,1,2)=r(w|\{u_1,x,u_3\})$.  Therefore,
$N(w)\cap[\Gamma_1(v)\cup\Gamma_2(v)]=\{u_1,z\}$. Since $G$ is
$3$-regular, $\Gamma_3(v)\neq\emptyset$. If $z\sim y$, then
$\{w,x\}$ is a separating set in $G$ which is impossible.
 Thus, $z$ has a neighbor in $\Gamma_3(v)$,
say $u$. If $u\nsim w$, then $d(w,u)=2=d(u_3,u)$ which implies
that $r(u_3|\{u_2,z,u\})=(2,1,2)=r(w|\{u_2,z,u\})$. Hence, $u\sim
w$ and it yields $r(w|\{u,v,x\})=r(z|\{u,v,x\})$. Consequently,
$N(v)$ is an independent set in $G$. }\end{proof}
\vspace{-.5cm}
\begin{figure}[ht]
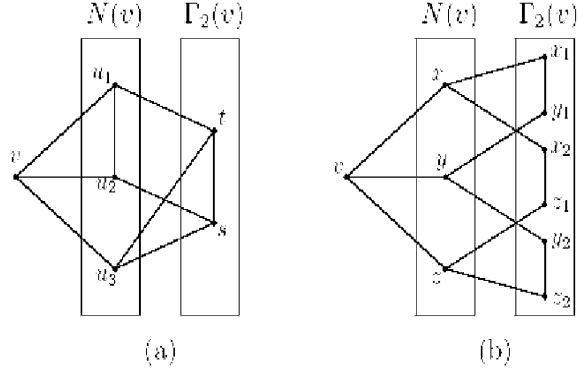
\hspace{3cm}
\vspace*{5cm}\special{em:graph fig3.bmp} \vspace*{1.5cm}
%%\caption{ $(a)$ $\Delta(G)=5$, $(b)$ Neighbors of a
%%vertex of degree $2$.\label{12}}
%%\end{figure}
%\caption{\label{34} The minimum degree of $G$ is more than $2$.}
%\end{figure}
\vspace{-1.2cm}\caption{\label{56} Two graphs with metric dimension $2$.}
\end{figure}
\begin{theorem}\label{G=K4} If $G$ is a randomly $3$-dimensional
graph, then $G=K_4$.
\end{theorem}
\begin{proof}{Suppose on the contrary that $G$ is a randomly
$3$-dimensional graph and $G\neq K_4$.  Let $v\in V(G)$ is an
arbitrary fixed vertex and $N(v)=\{x,y,z\}$. By
Proposition~\ref{N(v) independent}, $N(v)$ is an independent set
in $G$. Since $G$ is $3$-regular, there are six edges between
$N(v)$ and $\Gamma_2(v)$. If a vertex  $a\in\Gamma_2(v)$ is
adjacent to $x$ and $y$, then
$r(x|\{v,a,z\})=(1,1,2)=r(y|\{v,a,z\})$, which is impossible.
Therefore, by symmetry, each vertex of $\Gamma_2(v)$ has exactly
one neighbor in $N(v)$ and hence, $\Gamma_2(v)$ has exactly six
vertices.
%Let $H=\langle\Gamma_2(v)\rangle$.
If there exists a vertex $a\in \Gamma_2(v)$ with no neighbor in
$\Gamma_2(v)$, then by symmetry, let $a\sim z$. Thus,
$r(x|\{v,z,a\})=(1,2,3)=r(y|\{v,z,a\})$. Also, if there exists a
vertex $a\in \Gamma_2(v)$ with two neighbors $b$ and $c$ in
$\Gamma_2(v)$, by symmetry, let $a\sim z$, $b\nsim z$ and $c\nsim z$. Then,
%because by Proposition~\ref{N(v) independent},
%$N(a)$ is an independent set in $G$. Hence,
$r(b|\{v,z,a\})=(2,2,1)=r(c|\{v,z,a\})$. These contradictions
imply that $\Gamma_2(v)$ is a matching in $G$. Since all neighbors
of each vertex of $G$ constitute an independent set in $G$, the
induced subgraph $\langle\{v\}\cup N(v)\cup\Gamma_2(v)\rangle$ of
$G$ is as Figure~\ref{56}(b). Since $G$ is $3$-regular,
$\Gamma_3(v)\neq\emptyset$ and each vertex of $\Gamma_2(v)$ has
one neighbor in $\Gamma_3(v)$. Let $u\in\Gamma_3(v)$ be the
neighbor of $x_1$. Thus, $y_1\nsim u$. If $y_1$ and $z_2$ have no
common neighbor in $\Gamma_3(v)$, then
$r(x|\{x_1,u,z_2\})=(1,2,3)=r(y_1|\{x_1,u,z_2\})$. Therefore,
$y_1$ and $z_2$ have a common neighbor in $\Gamma_3(v)$, say $w$.
Consequently, $r(y|\{v,x,w\})=(1,2,2)=r(z|\{v,x,w\})$. This
contradiction implies that $G=K_4$. }\end{proof} The next
corollary characterizes all randomly $k$-dimensional graphs.
\begin{cor} Let $G$ be a graph with $\beta(G)=k>1$. Then, $G$ is a
randomly $k$-dimensional graph if and only if $G$ is a complete
graph $K_{k+1}$ or an odd cycle.
\end{cor}
%%%%%%%%%%%%%%%%%%%%%%%%%%%%%%%%%%%%%%%%%%%%%%%%%%%%%%%%%%%%%%%%%%%%%%%%%%%%%%%%%%%%%%%%%%%%%%%%%%%%%%%%%
%%%%%%%%%%%%%%%%%%%%%%%%%%%%%%%%%%%%%%%%%%%%%%%%%%%%%%%%%%%%%%%%%%%%%%%%%%%%%%%%%%%%%%%%%%%%%%%%%%%%%%%%%%%%%%
%%%%%%%%%%%%%%%%%%%%%%%%%%%%%%%%%%%%%%%%%%%%%%%%%%%%%%%%%%%%%%%%%%%%%%%%%%%%%%%%%%%%%%%%%%%%%%%%%%%%%%%%

\end{document}